\documentclass{amsart}%
\usepackage{amssymb}%
\usepackage{amsmath}%
\usepackage{amsfonts}%
\usepackage{graphicx}
\newtheorem{theorem}{Theorem}[section]

\newtheorem{proposition}[theorem]{Proposition}
\newtheorem{corollary}[theorem]{Corollary}
\newtheoremstyle{notauto}{}{}{\itshape}{}{\bfseries}{.}{0.5em}{\thmnote{#3}}
\theoremstyle{notauto}

\theoremstyle{definition}

\theoremstyle{remark}

\begin{document}
\title{Frobenius action on Carter subgroups}
\author{G\"{u}l\.{I}n Ercan$^{*}$}
\address{G\"{u}l\.{I}n Ercan, Department of Mathematics, Middle East echnical
University, Ankara, Turkey}
\email{ercan@metu.edu.tr}
\author{\.{I}sma\.{I}l \c{S}. G\"{u}lo\u{g}lu}
\address{\.{I}sma\.{I}l \c{S}. G\"{u}lo\u{g}lu, Department of Mathematics,
Do\u{g}u\c{s} University, Istanbul, Turkey}
\email{iguloglu@dogus.edu.tr}
\thanks{$^{*}$Corresponding author}
\maketitle

\begin{abstract}
Let $G$ be a finite solvable group and $H$ be a subgroup of $Aut(G)$. Suppose
that there exists an $H$-invariant Carter subgroup $F$ of $G$ such that the
semidirect product $FH$ is a Frobenius group with kernel $F$. We prove that
the terms of the Fitting series of $C_{G}(H)$ are obtained as the intersection
of $C_{G}(H)$ with the corresponding terms of the Fitting series of $G$, and
the Fitting height of $G$ may exceed the Fitting height of $C_{G}(H)$ by at
most one. As a corollary it is shown that for any set of primes $\pi$, the
terms of the $\pi$-series of $C_{G}(H)$ is obtained as the intersection of
$C_{G}(H)$ with the corresponding terms of the $\pi$-series of $G$, and the
$\pi$-length of $G$ may exceed the $\pi$-length of $C_{G}(H)$ by at most one.
They generalize the main results of \cite{Khu}.

\end{abstract}

\section{introduction}

Let $G$ and $A$ be finite groups such that $A$ acts on $G$ by automorphisms.
Both the structure of $A$ and the way it acts on $G$ has drastic consequences
on the structure of $G.$ As a typical result in this framework one can mention
the pioneering work of J.G.Thompson which says that a group $G$ having an
automorphism of prime order fixing no elements of $G$ except the identity, is nilpotent.

The present work is motivated by our research on possible generalizations of a
result due to Khukhro \cite{Khu} showing that some important group theoretic
invariants of a solvable group $G$ admitting a Frobenius group $A=FH$ of
automorphisms with Frobenius kernel $F$ and complement $H$ are closely related
to the corresponding invariants of the fixed point subgroup $C_{G}(H)$ if $F$
acts fixed-point-freely on $G.$ Under these conditions $F$ is a Carter
subgroup of the semidirect product $GF$ and $H$ acts on the solvable group
$GF$ leaving $F$ invariant and acts, not only fixed-point-freely, but also
Frobeniusly on $F$. Here we prove that almost the same result is true if $G$
is a solvable group, $H$ a group acting on $G$ and leaving a Carter subgroup
$F$ of $G$ invariant and acting Frobeniusly on it. Namely, the main results of
this paper are the following theorem on the Fitting series and its corollary
on $\pi$-series.

\begin{theorem}
Let $G$ be a finite solvable group and $H$ be a subgroup of $Aut(G)$. Suppose
that there exists an $H$-invariant Carter subgroup $F$ of $G$ such that the
semidirect product $FH$ is a Frobenius group with kernel $F$. Then we
have\newline

$(a)$ $F_{n}(C_{G}(H))=F_{n}(G)\cap C_{G}(H)$ for all $n\in\mathbb{N}$;

$(b)$ $h(G)\leq h(C_{G}(H))+1$. In fact $G=FF_{m}(G)$ where $m=h(C_{G}(H))$.
\end{theorem}

\begin{corollary}
Let $G$ be a finite solvable group and $H$ be a subgroup of $Aut(G)$. Suppose
that there exists an $H$-invariant Carter subgroup $F$ of $G$ such that the
semidirect product $FH$ is a Frobenius group with kernel $F$. Then we
have\newline

$(a)$ $O_{\pi}(C_{G}(H))=O_{\pi}(G)\cap C_{G}(H)$ for any set of primes $\pi$;

$(b)$ $O_{\pi_{1},\pi_{2},\dots,\pi_{k}}(C_{G}(H))=O_{\pi_{1},\pi_{2}%
,\dots,\pi_{k}}(G)\cap C_{G}(H)$ for any sets of primes\newline$\pi_{1}%
,\pi_{2},\dots,\pi_{k}$;

$(c)$ $\ell_{\pi}(G)\leq\ell_{\pi}(C_{G}(H))+1$.
\end{corollary}

As we have already pointed out the following result due to Khukhro follows as
a consequence. It should be noted that by \cite{Be} the condition $C_{G}(F)=1$
below directly implies the solvability of the group $G.$
\begin{corollary}

\cite{Khu} Let $G$ be a finite group
admitting a Frobenius group $FH$ of automorphisms with kernel $F$ and
complement $H$ such that $C_{G}(F)=1.$ Then we have\newline

$(a)$ $F_{n}(C_{G}(H))=F_{n}(G)\cap C_{G}(H)$ for all $n\in\mathbb{N}$;

$(b)$ $h(G)=h(C_{G}(H))$;

$(c)$ $O_{\pi}(C_{G}(H))=O_{\pi}(G)\cap C_{G}(H)$ for any set of primes $\pi$;

$(d)$ $O_{\pi_{1},\pi_{2},\dots,\pi_{k}}(C_{G}(H))=O_{\pi_{1},\pi_{2}%
,\dots,\pi_{k}}(G)\cap C_{G}(H)$ for any sets of primes\newline$\pi_{1}%
,\pi_{2},\dots,\pi_{k}$;

$(e)$ $\ell_{\pi}(G)=\ell_{\pi}(C_{G}(H))$.
\end{corollary}

This can be proven by regarding the Frobenius kernel $F$ as an $H$-invariant
Carter subgroup of the semidirect product $G_{1}=GF$ and $H$ as a group of
automorphisms of the group $G_{1}.$ Then the result follows immediately.

So far we have obtained several extensions of Khukhro's result replacing $FH$
by a Frobenius-like group with kernel $F$ under some mild additional
conditions (see \cite{EG1},\cite{EG2},\cite{EG3},\cite{EG4},\cite{EG5}%
,\cite{EG6},\cite{EG7}). Unfortunately, this time we must be satisfied with
the present form due to the example given in \cite{EG2}. Namely, there exists
a group $G$ admitting a group $H$ of automorphisms of prime order such that
$G=VQF$ where

\noindent$(i)$ $V=F(G)$ is an elementary abelian $p$-group for some prime $p$,
$Q$ is a $q$-group for some prime $q$ with $Q\unlhd QF$;\newline$(ii)$ $FH$ is
a Frobenius-like (but not Frobenius) group where $F$ is an extraspecial
$r$-group for some prime $r\notin\{p,q\}$, and $H$ centralizes $Z(F)$%
;\newline$(iii)$ $C_{Q}(F)=1, C_{V}(F)\ne1,$ but $C_{V}(F)H$ is a Frobenius
group;\newline$(iv)$ $C=C_{V}(F)\times F$ is a Carter subgroup of $G$ and $CH$
is a Frobenius-like group;\newline$(v)$ $Ker(C_{Q}(H) \,{on}\, C_{V}(H))\ne
Ker(C_{Q}(H) \,{on}\, V)=1$ and hence \newline$(vi)$ $F(C_{G}(H))\ne F(G)\cap
C_{G}(H).$

All groups are finite throughout the paper. The notation and the terminology
are standard as in \cite{Isaacs} except the following: The Fitting height and
the $\pi$-length of a group $G$ are denoted by $h(G)$ and $\ell_{\pi}(G)$, respectively.

\section{KEY PROPOSITION}

In this section we present the following result which makes the appearance of
the main result of this paper auxiliary. It should be pointed out that it is
of independent interest, too.

\begin{proposition}
Let a group $H$ act on the solvable group $G$ and let $F$ be an $H$-invariant
Carter subgroup of $G$ such that $FH$ is a Frobenius group with kernel $F$ and
complement $H$. Let $Q$ be a $q$-subgroup of $G$ such that $Q\unlhd GH$, and
$V$ be a $kGH$-module over a field of characteristic $p$ for distinct primes
$p$ and $q$ on which $Q$ acts nontrivially. Suppose that one of the following
holds:

\noindent (i) $[C_{V}(F),h]=C_{V}(F)$ for any nonidentity $h\in H;$

\noindent (ii) $chark$ is coprime to $|H|$.

\noindent Then we have $Ker(C_{Q}(H)\,{on}\,C_{V}(H))=Ker(C_{Q}(H)\,{on}\,V).$
\end{proposition}
\begin{proof}
Suppose the proposition is false and choose a counterexample with minimum
$dim_{k}V+\left\vert QFH\right\vert $. We split the proof into a sequence of
steps. To simplify the notation we set $K=Ker(C_{Q}(H)\,on\,C_{V}%
(H))$.\newline

\textit{(1) We may assume that $G=QF$. Furthermore $Q\not \leq F$, in
particular $[Q,F]\neq1$, and hence $O_{q^{\prime}}(F)\neq1.$}

\begin{proof}
It can be easily seen by induction that $G=QF.$ We may also assume that
$Q\not \leq F$, in particular $[Q,F]\neq1$, because otherwise $C_{Q}(H)=1$ and
the theorem follows. Note that if $QF$ is a $q$-group then $F$ is properly
contained in $QF$ as $Q\not\leq F.$ But then $F$ is properly contained in ts
normalizer in $QF$, which is not the case.
\end{proof}

\textit{(2) We may assume that $k$ is a splitting field for all subgroups of
$GH$.}

\begin{proof}
We consider the $GH$-module $\bar{V}=V\otimes_{k}{\bar{k}}$ where $\bar{k}$ is
the algebraic closure of $k.$ Notice that $C_{\bar{V}}(H)=C_{V}(H)\otimes
_{k}{\bar{k}}$. Therefore once the proposition has been proven for the group
$GH$ on $\bar{V}$, it becomes true for $GH$ on $V$ also.
\end{proof}

\textit{(3) $V$ is an irreducible $GH$-module on which $G$ acts faithfully.}

\begin{proof}
Let $W=X/Y$ be a $GH$-composition factor of $V$ on which $K$ acts
nontrivially. If $C_{V}(F)=1$ then $C_{Y}(F)=1$ and so $C_{W}(H)=C_{X}(H)Y/Y$
by Theorem 1.5 in \cite{Khu}. Otherwise by hypothesis the group $C_{V}(F)H$ is
Frobenius and hence $chark$ is coprime to $|H|$ which implies that
$C_{W}(H)=C_{X}(H)Y/Y$. Notice that if $W\neq V$ then $Ker(C_{Q}(H)$ on
$C_{W}(H))=Ker(C_{Q}(H)$ on $W)$ holds by induction. Hence $K=Ker(K$ on
$C_{W}(H))=Ker(K$ on $W)$ which contradicts the assumption that $K$ acts
nontrivially on $W.$ Therefore we can regard $V$ as an irreducible $GH$-module.

We set next $\bar{G}=G/Ker(G$ on $V)$ and consider the action of the group
$\bar{G}H$ on $V$. An induction argument gives $Ker(C_{\bar{Q}}(H)$ on
$C_{V}(H))=Ker(C_{\bar{Q}}(H)$ on $V)$ and hence $Ker(\overline{C_{Q}(H)}$ on
$C_{V}(H))=Ker(\overline{C_{Q}(H)}$ on $V)$ which leads to a contradiction.
Thus we may assume that $G$ acts faithfully on $V.$
\end{proof}

\vspace{1mm} \noindent It should be noted that we need only to prove $K=1$ due
to the faithful action of $Q$ on $V$. So we assume this to be false.
\vspace{1mm}

\textit{(4) $L=K\cap Z(C_{Q}(H))\neq1$ since $1\neq K\trianglelefteq C_{Q}%
(H)$. Pick $1\neq c\in L$ of order $q$. Then $Q=\left\langle c^{F}%
\right\rangle $.}

\begin{proof}
Suppose first that $Q\neq\left\langle c^{F}\right\rangle $. An induction
argument applied to the action of $\left\langle c^{F}\right\rangle FH$ on $V$
we get $c\in Ker(C_{\left\langle c^{F}\right\rangle }(H)\,on\,C_{V}%
(H))=Ker(C_{\left\langle c^{F}\right\rangle }(H)\,on\,V)=1.$ This
contradiction gives $Q=\left\langle c^{F}\right\rangle $.
\end{proof}

\vspace{1mm} By Clifford's theorem the restriction of the $GH$-module $V$ to
the normal subgroup $Q$ is a direct sum of $Q$-homogeneous components. Let
$\Omega$ denote the set of all $Q$-homogeneous components of $V$. \vspace{1mm}

\textit{(5) $K$ acts trivially on the sum of components in any regular
$H$-orbit in $\Omega$. Therefore there exists $W\in\Omega$ such that
$Stab_{H}(W)\neq1$.}

\begin{proof}
Let $U$ be an element in $\Omega$ such that $\{U^{y}:y\in H\}$ is a regular
$H$-orbit in $\Omega$ and let $X$ be the sum of components. Then $K$ acts
trivially on $C_{X}(H)=\{\Sigma_{y\in H}v^{y}:v\in U\}$ and hence trivially on
$X$.
\end{proof}

\textit{ (6) $F$ acts transitively on $\Omega$ and $H$ fixes an element, say
$W$, of $\Omega$ and $H$ acts regularly and $K$ acts trivially on the set
$\Omega\setminus\{W\}.$}
\begin{proof}
By (5) there exists $W\in\Omega$ such that $Stab_{H}(W)\neq1$. Let $\Omega
_{1}$ be the $F$-orbit on $\Omega$ containing $W$. Clearly, we have
$Stab_{H}(W)\leq H_{1}=Stab_{H}(\Omega_{1}).$ So $H_{1}\neq1$.

The group $H$ acts transitively on $\left\{  \Omega_{i}:i=1,2,\ldots
,s\right\}  ,$ the collection of $F$-orbits on $\Omega$. Let now
$V_{i}=\bigoplus_{W\in\Omega_{i}}W$ for $i=1,2,\ldots,s.$ Suppose that $s>1.$
Then $H_{1}$ is a proper subgroup of $H.$ Applying induction to the action of
$GH_{1}$ on $V_{1}$ we obtain

$Ker(C_{Q}(H_{1})$ on $C_{V_{1}}(H_{1}))=Ker(C_{Q}(H_{1})$ on $V_{1})$.

It follows that $Ker(C_{Q}(H)$ on $C_{V_{1}}(H_{1}))=Ker(C_{Q}(H)$ on $V_{1})$
holds as $C_{Q}(H)\leq C_{Q}(H_{1}).$ On the other hand we have $C_{V}%
(H)=\left\{  u^{x_{1}}+u^{x_{2}}+\cdots+u^{x_{s}}:u\in C_{V_{1}}%
(H_{1})\right\}  $ where $x_{1},\ldots,x_{s}$ is a complete set of right coset
representatives of $H_{1}$ in $H$. By definition $K$ acts trivially on
$C_{V}(H)$ and normalizes each $V_{i}$. Then $K$ is trivial on $C_{V_{1}%
}(H_{1})$ and hence on $V_{1}.$ As $K$ is normalized by $H$ we see that $K$
acts trivially on each $V_{i}$ and hence on $V.$ This contradiction shows that
$F$ acts transitively on $W$ and hence $H_{1}=H.$

Let now $S=Stab_{FH}(W)$ and $F_{1}=F\cap S$. Then $\left\vert F:F_{1}%
\right\vert =\left\vert \Omega\right\vert =\left\vert FH:S\right\vert .$
Notice next that $(\left\vert S:F_{1}\right\vert ,\left\vert F_{1}\right\vert
)=1$ as $(\left\vert F\right\vert ,\left\vert H\right\vert )=1$. Let $S_{1}$
be a complement of $F_{1}$ in $S.$ Then we have $\left\vert F:F_{1}\right\vert
=\left\vert F\right\vert \left\vert H\right\vert /\left\vert F_{1}\right\vert
\left\vert S_{1}\right\vert $ which implies that $\left\vert H\right\vert
=\left\vert S_{1}\right\vert .$ Therefore we may assume that $S=F_{1}H,$ that
is $W$ is $H$-invariant.

Finally let $x\in F$ and $1\neq h\in H$ such that $(W^{x})^{h}=W^{x}$ holds.
Then $[h,x^{-1}]\in F_{1}$ and so $F_{1}x=F_{1}x^{h}=(F_{1}x)^{h}$ implying
the existence of an element $g\in F_{1}x\cap C_{F}(h)$ by Theorem 3.27 in
\cite{Isaacs}. Now the Frobenius action of $H$ on $F$ gives that $x\in F_{1}.$
This means that $W$ is the only element in $\Omega$ which is stabilized by
some nonidentity element of $H$ and hence all the orbits of $H$ on $\Omega$
except $\{W\}$ are regular.
\end{proof}

\textit{(7) $C_{Z(Q)}(F)=1$. }

\begin{proof}
Due to the scalar action of $Z(Q/C_{Q}(W))$ on $W$, we have $[Z(Q),H]\leq
C_{Q}(W)$ and hence $[C_{Z(Q)}(F),H]\leq C_{Q}(W)$. As $[C_{Z(Q)}(F),H]$ is
centralized by $F$, $[C_{Z(Q)}(F),H]\leq\bigcap_{f\in F}C_{Q}(W^{f}%
)=C_{Q}(V)=1.$ It follows that $C_{Z(Q)}(F)\leq N_{G}(F)\cap C_{Q}(H)\leq
C_{F}(H)=1.$
\end{proof}

\textit{(8) Final Contradiction.}

\begin{proof}
By (4), $Q=\left\langle c^{F}\right\rangle $. It follows by (6) that for any
$x\in F\setminus F_{1},$ $c^{x}$ centralizes $W$. Thus we have $Q=\left\langle
c^{F_{1}}\right\rangle C_{Q}(W)$.

Let $m$ be the nilpotency class of $Q/C_{Q}(W)$ which is of course equal to
the nilpotency class of $Q/C_{Q}(W^{f})$ for any $f\in F.$ As $Q=Q/\cap_{f\in
F}C_{Q}(W^{f})$ is isomorphic to a subgroup of $\oplus_{f\in F}Q/C_{Q}(W^{f})$
we see that $m$ is equal to the nilpotency class of $Q$. Then there exists
$z=[c^{y_{1}},\ldots,c^{y_{m}}]$ where $y_{i}\in F_{1},i=1,\ldots,m$ such that
$z\notin C_{Q}(W)$. Clearly $z\in Z(Q).$ Notice that for all $x\in F\setminus
F_{1},$ we have $z^{x}=[c^{y_{1}x},\ldots,c^{y_{m}x}]\in C_{Q}(W).$
Furthermore we have $[Z(Q),F_{1}]\leq C_{Q}(W)$ due to the scalar action of
$Z(Q/C_{Q}(W))$ on $W$. Let $X=O_{q^{\prime}}(F).$ By (7) $C_{Z(Q)}(X)=1$.
Then,
\[
1=\prod_{f\in X}z^{f}=(\prod_{f\in X\setminus F_{1}}z^{f})(\prod_{f\in X\cap
F_{1}}z^{f})\in(\prod_{f\in X\cap F_{1}}z^{f})C_{Q}(W)=z^{\left\vert X\cap
F_{1}\right\vert }C_{Q}(W).
\]
implying that $q$ divides $\left\vert X\cap F_{1}\right\vert $, which is
impossible. This completes the proof.
\end{proof}

The following is an important consequence of the above proposition and appears
as another version of Proposition 4.1 in \cite{EG} showing that the condition
that $C_{V}(F)=1$ can be replaced by the condition that $C_{Q}(F)=1$ without
assuming any coprimeness condition.

\begin{corollary}
Let $FH$ be a Frobenius group with kernel $F$ and complement $H$ acting on a
$q$-group $Q$ for some prime $q$. Let $V$ be a $kQFH$-module over a field of
characteristic is coprime to $|QH|$. If $C_{Q}(F)=1$ then we have
\[
Ker(C_{Q}(H)\,{on}\,C_{V}(H))=Ker(C_{Q}(H)\,{on}\,V).
\]

\end{corollary}

\begin{proof}
As $C_{Q}(F)=1$, we can regard $F$ as an $H$-invariant Carter subgroup ofthe
semidirect product $QF$. Then we appeal to the above proposition by letting
$G=QF.$
\end{proof}

\bigskip

\section{PROOFS\ OF\ THE\ MAIN\ RESULTS}

The following proposition will be needed in proving Theorem 1.1 and Corollary 1.2.

\begin{proposition}
Let $G$ be a group and $H$ be a subgroup of $Aut(G)$. Suppose that there
exists an $H$-invariant Carter subgroup $F$ of $G$ such that the semidirect
product $FH$ is a Frobenius group with kernel $F$. For any $H$-invariant
solvable normal subgroup $N$ of $G$ we have $C_{G/N}(H)=C_{G}(H)N/N.$
\end{proposition}

\begin{proof}
We proceed by induction on the order of $G.$ As $N$ is solvable there exists a
prime $p$ such that $O_{p}(N)\neq1.$ Set $C_{O_{p}(N)}(F)=1=G/O_{p}(N).$ We
first observe that $C_{\overline{G}}(H)=\overline{C_{G}(H)}:$ This follows
from Theorem 1.5 in \cite{Khu} in case $C_{O_{p}(N)}(F)=1.$ Otherwise $1\neq
C_{O_{p}(N)}(F)\leq F$ which implies that $p$ divides $|F|.$ Then $p$ is
coprime to $|H|$ and the claim follows.

As $\overline{F}$ is an $H$-invariant Carter subgroup of $\overline{G}$, an
induction argument gives that $C_{\overline{G}/\overline{N}}(H)=C_{\overline
{G}}(H)\overline{N}/\overline{N}.$ Notice that
\[
C_{G/N}(H)\cong C_{\overline{G}/\overline{N}}(H)=C_{\overline{G}}%
(H)\overline{N}/\overline{N}\cong C_{G}(H)N/N
\]
which proves the claim.
\end{proof}

\textit{\textbf{Proof of Theorem 1.1}} \thinspace(a) To prove the result for
$n=1$ we use induction on the order of $G.$ Set $\overline{G}=G/F(G)$. Then
$\overline{F}$ is an $H$-invariant Carter subgroup of $\overline{G}$ such that
the semidirect product $\overline{F}H$ is a Frobenius group with kernel
$\overline{F}$. Then the result holds for $\overline{G}$ by induction, that is
$F(C_{\overline{G}}(H))=F(\overline{G})$. In particular $\overline
{F(C_{G}(H))}\leq F(C_{\overline{G}}(H))=F(\overline{G})$ by Proposition 3.1,
implying that $F(C_{G}(H))\leq F_{2}(G).$

Notice that if $F(C_{G}(H))$ is not contained in $F(G)$ then there exists a
prime $q$ such that $Q_{0}=O_{q}(C_{G}(H))$ is not contained in $F(G)$. Let
$Q\in Syl_{q}(F_{2}(G))$ such that $Q_{0}\leq Q.$ So there exists a prime
$p\neq q$ such that $[V,Q_{0}]\neq1$ where $V=O_{p}(G)/\Phi(O_{p}(G))$. We
have $C_{V}(H)=C_{O_{p}(F(G))}(H)\Phi(O_{p}(G))/\Phi(O_{p}(G))$ by Proposition
3.1. Now by Proposition 2.1 applied to the action of $GH$ on $V$ we get
\[
Ker(C_{Q}(H)\,{on}\,C_{V}(H))=Ker(C_{Q}(H)\,{on}\,V).
\]
and hence $Q_{0}$ centralizes $V,$ which is a contradiction. So we have the
result for $n=1$.

Suppose that the result is true for a fixed but arbitrary $k$, that is,
$F_{k}(C_{G}(H))=F_{k}(G)\cap C_{G}(H).$ Set $\overline{G}=G/F_{k}(G).$ Now
using Proposition 3.1 and the induction assumption we get
\[
F_{k+1}(C_{G}(H))/F_{k}(C_{G}(H))\cong F(C_{\overline{G}}(H))=C_{F(\overline
{G})}(H)=C_{\overline{F_{k+1}(G)}}(H)=
\]%
\[
\overline{C_{F_{k+1}(G)}(H)}\cong C_{F_{k+1}(G)}(H)/C_{F_{k}(G)}%
(H)=C_{F_{k+1}(G)}(H)/F_{k}(C_{G}(H)
\]
and the result follows for $k+1.$ This completes the proof of part (a).

(b) We proceed by induction on the order of $G$. Set $\overline{G}=G/F(G)$.
Then $\overline{F}$ is an $H$-invariant Carter subgroup of $\overline{G}$ such
that the semidirect product $\overline{F}H$ is a Frobenius group with kernel
$\overline{F}$. Therefore the theorem is true for $\overline{G}$ by induction,
that is,
\[
h(G)-1=h(\overline{G})\leq h(C_{\overline{G}}(H))+1=h(C_{G}(H)/C_{F(G)}(H))+1.
\]
By (a) we have $C_{F(G)}(H)=F(C_{G}(H)).$ Then $h(G)-1=h(\overline{G})\leq
h(C_{G}(H))$ and so $h(G)\leq h(C_{G}(H))+1$, as desired. Let $m=h(C_{G}(H)).$
Then $G/F_{m}(G)$ is nilpotent and hence is covered by the homomorphic image
of $F$ as the Carter subgroup of a nilpotent group is not proper. This shows
that $G=F_{m}(G)F.$ $\Box$ 
\vspace{4mm}

\textit{\textbf{Proof of Corollary 1.2}} (a)\thinspace\ Clearly, we have
$O_{\pi}(G)\cap C_{G}(H)\leq O_{\pi}(C_{G}(H)).$ Suppose that $O_{\pi}%
(C_{G}(H))\neq1$ . Then there exists a prime $p\in\pi$ such that $O_{p}%
(C_{G}(H))\neq1.$ As $O_{p}(C_{G}(H))\leq F(C_{G}(H))\leq F(G)$ by Theorem
1.1, we see that $O_{p}(G)\neq1.$ Set $\overline{G}=G/O_{p}(G).$ Since
$|\overline{G}|\leq|G|$ we get by induction that $O_{\pi}(C_{\overline{G}%
}(H))\leq O_{\pi}(\overline{G})\leq\overline{O_{\pi}(C_{G}(H))}.$ As
$C_{\overline{G}}(H)=\overline{C_{G}(H)}$ by Proposition 3.1 and $O_{\pi
}(C_{\overline{G}}(H))=O_{\pi}(\overline{C_{G}(H)})$ we obtain $O_{\pi}%
(C_{G}(H))\leq O_{\pi}(G)\cap C_{G}(H)$, and the claim follows. Now part (b)
follows immediately from (a) by induction.

To prove part (c) we proceed by induction on the order of $G$. Let $\ell_{\pi
}(C_{G}(H))=m$ and $M=O_{{\pi}^{\prime},\pi,{\pi}^{\prime},\ldots,\pi,{\pi
}^{\prime}}(G)$ where the number of $\pi$' s is equal to $m.$ By (b) we have
$M\cap C_{G}(H)=O_{{\pi}^{\prime},\pi,{\pi}^{\prime},\ldots,\pi,{\pi}^{\prime
}}(C_{G}(H))=C_{G}(H)$ and hence $C_{G}(H)\leq M$ due to coprimeness. That is
$C_{G/M}(H)=1$. By Lemma 1.3 in \cite{Khu} it follows that $[G,F]\leq M$ and
so $G/M\leq N_{G/M}(FM/M)=FM/M.$ That is $G/M$ is nilpotent and hence is a
$\pi$-group. This shows that $\ell_{\pi}(G)\leq m+1$ as claimed. $\Box$

\end{proof}

\end{document}